\documentclass{amsart}
\usepackage{amsmath}
\usepackage{amssymb}
\usepackage{mathrsfs}
\usepackage{pifont}
\usepackage{amsfonts}
\usepackage{stmaryrd}
\usepackage{latexsym,amssymb,amsmath}
\usepackage[all]{xy}
\newtheorem{theorem}{Theorem}[section]
\newtheorem{lemma}[theorem]{Lemma}
\newtheorem{corollary}[theorem]{Corollary}
\newtheorem{proposition}[theorem]{Proposition}

\theoremstyle{definition}
\newtheorem{definition}[theorem]{Definition}

\theoremstyle{remark}

\numberwithin{equation}{section}

\begin{document}

\title[Hom-Lie superalgebra structures]{Hom-Lie superalgebra structures on finite-dimensional simple Lie superalgebras}
\author{Bintao Cao}
\address{School of Mathematics and Computational Science, Sun Yat-sen University,
Guangzhou, 510275, China}
\email{caobt@mail.sysu.edu.cn}
\thanks{The first author is supported by NSFC Grant 11101436.}

\author{Li Luo}
\address{Department of Mathematics, East China Normal University, Shanghai, 200241, China}
\email{lluo@math.ecnu.edu.cn}
\thanks{The second author is supported by the NSFC Grant 11101151 and the Fundamental Research Funds for the Central Universities.}





\begin{abstract}
Hom-Lie superalgebras, which can be considered as a deformation of
Lie superalgebras, are $\mathbb{Z}_2$-graded generalization of
Hom-Lie algebras. In this paper, we prove that there is only the
trivial Hom-Lie superalgebra structure over a finite-dimensional
simple Lie superalgebra.
\end{abstract}

\maketitle



\section{introduction}

The notion of Hom-Lie algebras was introduced by Hartwig, Larsson
and Silverstrov \cite{HLS}. Part of the reason is to study the
$q$-deformations of the Witt and the Virasoro algebras. In fact, the
works \cite{AS,H,L} have already given the early forerunner of
Hom-Lie algebras, which also studied the $q$-deformations of the
Witt and the Virasoro algebras. We refer to \cite{CKL,CILPP,CPP} for
the other pioneering works of this field.

Because of close relation to discrete and deformed vector fields and
differential calculus, Hom-Lie algebras have recently drawn
considerable attention in \cite{BM,JL,MS,Sh, Y1,Y2,Y3,ZHB}. In
particular, Jin and Li \cite{JL} proved that such an algebra
structure over a finite-dimensional simple Lie algebra is equivalent
to the trivial one. They also determined the isomorphic classes of
nontrivial Hom-Lie algebra structures over finite-dimensional
semi-simple Lie algebras.

Recently, Hom-Lie algebras were generalized to Hom-Lie superalgebras
by Ammar and Makhlouf \cite{AM} and to Hom-Lie color algebras by
Yuan \cite{Y}.

\begin{definition}
A \emph{Hom-Lie superalgebra} is a triple $(\mathfrak{g}, [-,-],
\sigma)$ consisting of a $\mathbb{Z}_2$-graded vector space
$\mathfrak{g}$, a bilinear map $[-,-]:\mathfrak{g}\times
\mathfrak{g}\rightarrow \mathfrak{g}$ and an even linear map
$\sigma: \mathfrak{g}\rightarrow \mathfrak{g}$ satisfying
\begin{equation}
\sigma[x,y]=[\sigma(x),\sigma(y)]\quad\mbox{(multiplicativity)},
\end{equation}
\begin{equation}
[x,y]=-(-1)^{|x||y|}[y,x]\quad\mbox{(graded skew-symmetry)}
\end{equation}
and
\begin{equation}\label{Jacobi}
(-1)^{|x||z|}[\sigma(x),[y,z]]+(-1)^{|y||x|}[\sigma(y),[z,x]]+(-1)^{|z||y|}[\sigma(z),[x,y]]=0
\end{equation}
$$\quad\quad\mbox{(graded Hom-Jacobi identity, or call graded $\sigma$-twisted Jacobi identity)},$$
where $x$, $y$ and $z$ are homogeneous elements in $\mathfrak{g}$.
\end{definition}

\begin{definition}
 For any Lie superalgebra $\mathfrak{g}$, denote its Lie bracket by $[-,-]$
and take an even linear map $\sigma: \mathfrak{g}\rightarrow
\mathfrak{g}$. We say $(\mathfrak{g},\sigma)$ is a \emph{Hom-Lie
superalgebra structure over the Lie superalgebra $\mathfrak{g}$} if
$(\mathfrak{g}, [-,-], \sigma)$ is a Hom-Lie superalgebra.
\end{definition}

It is clear that the Hom-Lie superalgebra
$(\mathfrak{g},[-,-],\mbox{id})$ is the Lie superalgebra
$\mathfrak{g}$ itself. We call $(\mathfrak{g},\mbox{id})$ \emph{the
trivial Hom-Lie superalgebra structure of $\mathfrak{g}$}. If
$(\mathfrak{g},\sigma)$ is a Hom-Lie superalgebra structure on the
Lie superalgebra $\mathfrak{g}$, then $\sigma$ has to be a
homomorphism of $\mathfrak{g}$ because of the multiplicativity
(2.1). For any simple Lie superalgebra $\mathfrak{g}$, a
homomorphism $\sigma: \mathfrak{g}\rightarrow \mathfrak{g}$ is
either an automorphism or $\sigma=0$. If $\sigma=0$, then the
Hom-Lie superalgebra $(\mathfrak{g},\sigma)$ is too simple to say
anything since the graded Hom-Jacobi identity (2.3) means nothing
but ``$0=0$''. Therefore, we always suppose that $\sigma$ is an
automorphism of $\mathfrak{g}$ in the sequel.

See \cite{JL} for the definition of Hom-Lie algebra structures over
a Lie algebra. We point out that they made a slight change for the
definition of Hom-Lie algebras so that in their case both
$(\mathfrak{g},\mbox{id})$ and $(\mathfrak{g},0)$ are the Lie
algebra $\mathfrak{g}$ itself.

Influencing by the work \cite{JL}, it is nature to consider whether
there exist nontrivial Hom-Lie super algebra structures over
finite-dimensional simple Lie superalgebras. The main result of this
present work is as follows.

\vspace{0.3cm} \noindent{\bf Main Theorem:} The Hom-Lie superalgebra
structures on a finite-dimensional simple Lie superalgebra
$\mathfrak{g}$ are trivial. That is, if $(\mathfrak{g},\sigma)$ is a
Hom-Lie superalgebra then $\sigma=\mbox{id}$. \vspace{0.3cm}

It is well known that there are two families of finite-dimensional
Lie superalgebras. One is called classical Lie superalgebras and the
other is called Cartan-type Lie superalgebras. For the classical
ones $\mathfrak{g}=\mathfrak{g}_{\bar{0}}\oplus
\mathfrak{g}_{\bar{1}}$, we first restrict the Hom-Lie superalgebra
structure $(\mathfrak{g},\sigma)$ to the even part
$\mathfrak{g}_{\bar{0}}$ so that we can use the result in \cite{JL}
to get that $\sigma|_{\mathfrak{g}_{\bar{0}}}=\mbox{id}$, and then
combine the description of the automorphisms for classical Lie
superalgebras (c.f. \cite{S,GP}) to check that $\sigma=\mbox{id}$
case by case. For the Cartan-type ones, we first show that the
isomorphism must be identity on the subspace
$\mathfrak{g}_{-1}\bigoplus \mathfrak{g}_{0}$, via the transitivity
of $W(n)$, $S(n)$ or $H(n)$. Then the conclusion is obtained by the
nontriviality and the  irreducibility of $\mathfrak{g}_{-1}$ as a
$\mathfrak{g}_{0}$-module.

The paper is organized as follows. In section 2, we recall the
definition of Hom-Lie superalgebras and list some endomorphisms of
$\mathfrak{gl}(m|n)$, which are used to describe the automorphism
groups for classical simple Lie superalgebras. We also list some
properties of Cartan type Lie superalgebras, which will be used to
prove the main theorem for these type cases in Section 4. Section 3
is devoted to the proof of the main theorem for classical Lie
superalgebras while Section 4 is devoted to those for Cartan-type
Lie superalgebras.

Throughout this paper, the basic field is assumed to be the complex
number field $\mathbb{C}$. The symbol
$\mathbb{Z}_2=\{\bar{0},\bar{1}\}$ will stand for the group of two
elements. When $|x|$ appears in an expression, we will always
implicitly regard $x$ as a $\mathbb{Z}_2$-homogeneous element and
automatically extend the relevant formulae by linearity (whenever
applicable). For any superalgebra in this paper, the homomorphisms
always mean even homomorphisms.

\section{Preliminaries}
In this section, we first list some endomorphisms of
$\mathfrak{gl}(m|n)$, which are used to describe the automorphism
groups for classical simple Lie superalgebras in the Section 3. We
also list some properties of Cartan type Lie superalgebras, which
will be used to prove the Cartan type cases of the main theorem in
the Section 4.

\subsection{Some endomorphisms of
$\mathfrak{gl}(m|n)$} \

Here we list some endomorphisms for the general linear Lie
superalgebra
\begin{equation}
\mathfrak{gl}(m|n):=\left\{\left(\begin{array}{cc}A&B\\C&D
\end{array}\right)|A\in M_m, B\in M_{m,n},C\in M_{n,m},D\in M_n\right\},
\end{equation} where $M_{p,q}$ is the set of all $p\times q$
matrices and $M_{p}:=M_{p,p}$.

For any $(X,Y)\in SL_m\times SL_n$, define $\mbox{Ad}(X,Y)\in
\mbox{End } \mathfrak{gl}(m|n)$ by
\begin{equation}
\mbox{Ad}(X,Y): \left(\begin{array}{cc}A&B\\C&D
\end{array}\right)\mapsto\left(\begin{array}{cc}XAX^{-1}&XBY^{-1}\\YCX^{-1}&YDY^{-1}
\end{array}\right).
\end{equation} This implies a group homomorphism $\mbox{Ad}: SL_m\times
SL_n\rightarrow\textbf{Aut}\mathfrak{gl}(m|n)$. That is,
\begin{equation}
\mbox{Ad}(X_1,Y_1)\mbox{Ad}(X_2,Y_2)=\mbox{Ad}(X_1X_2,Y_1Y_2)\quad\mbox{for
any $(X_1,Y_1), (X_2,Y_2)\in SL_m\times SL_n$}.
\end{equation}

For any $\lambda\in\mathbb{C}^\times:=\mathbb{C}\setminus\{0\}$,
there is also an endomorphism defined by
\begin{equation}
\jmath(\lambda):\left(\begin{array}{cc}A&B\\C&D
\end{array}\right)\mapsto\left(\begin{array}{cc}A&\lambda B\\\lambda^{-1}C&D
\end{array}\right).
\end{equation}
It is clear that
\begin{equation}
\jmath(\lambda_1)\jmath(\lambda_2)=\jmath(\lambda_1\lambda_2)\quad\mbox{for
any $\lambda_1,\lambda_2\in\mathbb{C}^\times$}.
\end{equation}

The \emph{supertransposition} $\tau$ is the endomorphism of
$\mathfrak{gl}(m|n)$ given by
\begin{equation}
\tau:\left(\begin{array}{cc}A&B\\C&D
\end{array}\right)\mapsto\left(\begin{array}{cc}-A^t&C^t\\-B^t&-D^t
\end{array}\right).
\end{equation}
It satisfies that
\begin{equation}
\tau^2=\jmath(-1)\quad\mbox{and}\quad\tau^4=1.
\end{equation}

Furthermore, when $m=n$, there is another endomorphism $\pi$ of
$\mathfrak{gl}(n|n)$ given by
\begin{equation}
\pi:\left(\begin{array}{cc}A&B\\C&D
\end{array}\right)\mapsto\left(\begin{array}{cc}D&C\\B&A
\end{array}\right),
\end{equation}
which satisfies that
\begin{equation}
\pi^2=1.
\end{equation}

At last, for any $\left(\begin{array}{cc}a&b\\c&d
\end{array}\right)\in SL_2$, we define the endomorphism of
$\mathfrak{gl}(2|2)$ as follows,
\begin{equation}
\rho\left(\begin{array}{cc}a&b\\c&d
\end{array}\right):\left(\begin{array}{cc}A&B\\C&D
\end{array}\right)\mapsto\left(\begin{array}{cc}A&aB+b\Psi(C)\\c\Psi(B)+dC&D
\end{array}\right),
\end{equation} where $\Psi(F)=\left(\begin{array}{cc}0&1\\-1&0
\end{array}\right)F^t\left(\begin{array}{cc}0&1\\-1&0
\end{array}\right)$. One can check that $\rho:SL_2\rightarrow\textbf{Aut}\mathfrak{gl}(2|2)$ is a group homomorphism. That is,
\begin{equation}
\rho(AB)=\rho(A)\rho(B) \quad\mbox{for any $A,B\in SL_2$}.
\end{equation}

Moreover, we can check easily that for any $(X,Y)\in SL_m\times
SL_n$ and $\lambda\in\mathbb{C}^\times$
\begin{equation}
\mbox{Ad}(X,Y)\jmath(\lambda)=\jmath(\lambda)\mbox{Ad}(X,Y),
\end{equation}
\begin{equation}
\mbox{Ad}(X,Y)\tau=\tau\mbox{Ad}((X^t)^{-1},(Y^{t})^{-1})
\end{equation} and
\begin{equation}
\jmath(\lambda)\tau=\tau\jmath(\lambda^{-1}).
\end{equation}
In the case of $m=n$,
\begin{equation}
\mbox{Ad}(X,Y)\pi=\pi\mbox{Ad}(Y,X),
\end{equation}
\begin{equation}
\jmath(\lambda)\pi=\pi\jmath(\lambda^{-1})
\end{equation} and
\begin{equation}
\tau\pi=\pi\tau\jmath(-1)=\pi\tau^3.
\end{equation}
Finally if $m=n=2$, then for any $\left(\begin{array}{cc}a&b\\c&d
\end{array}\right)\in SL_2$,
\begin{equation}
\mbox{Ad}(X,Y)\rho\left(\begin{array}{cc}a&b\\c&d
\end{array}\right)=\rho\left(\begin{array}{cc}a&b\\c&d
\end{array}\right)\mbox{Ad}(X,Y)
\end{equation}
and
\begin{equation}
\rho\left(\begin{array}{cc}a&b\\c&d
\end{array}\right)\pi=\pi\rho\left(\begin{array}{cc}d&c\\b&a
\end{array}\right).
\end{equation}

\subsection{Some properties for Cartan type Lie superalgebras} \

Let $\Lambda(n)$ be the exterior algebra in $n$ indeterminates
$\xi_1,\xi_2,\ldots,\xi_n$. There are four simple Lie superalgebras
consisting of its derivations, which are listing below:

\begin{equation}\label{4.1}
  W(n):=\{\sum_{j=1}^nf_j\frac{\partial}{\partial{\xi_j}}|
  f_j\in\Lambda(n)\}\quad (n\geq3);
\end{equation}
\begin{equation}\label{4.2}
  S(n):=\{\sum_{j=1}^nf_j\frac{\partial}{\partial{\xi_j}}|
  f_j\in\Lambda(n),\ \sum_{j=1}^n\frac{\partial
  f_j}{\partial\xi_j}=0\}\quad(n\geq3);
\end{equation}
\begin{eqnarray}\label{4.3}
  \widetilde{S}(n):=\{(1-\xi_1\xi_2\cdots\xi_n)\sum_{j=1}^nf_j\frac{\partial}{\partial{\xi_j}}|
  f_j\in\Lambda(n),\ \sum_{j=1}^n\frac{\partial
  f_j}{\partial\xi_j}=0\}\\\nonumber (n\geq 4 \mbox{ is an even number});
\end{eqnarray}and
\begin{equation}\label{4.4}
  H(n):=\{\sum_{j=1}^n\frac{\partial
  f_j}{\partial\xi_j}\frac{\partial}{\partial\xi_j}|
  f_j\in\Lambda(n)\}\quad (n\geq 4).
\end{equation}

There is a natural $\mathbb{Z}$-gradation
$W(n)=\bigoplus_{j=-1}^{n-1}W(n)_j$ via setting $\mbox{deg}\xi_i=1$
and $\mbox{deg}\frac{\partial}{\partial\xi_i}=-1$ for any
$i=1,2,\ldots,n$, which is called the \emph{principal gradation}. It
induces the $\mathbb{Z}_2$-gradation on $W(n)$:
\begin{equation}\label{2.1}
W(n)_{\bar{s}}:=\bigoplus_{j\equiv s(\mbox{mod}\ 2)}W(n)_j\ \ \ \
\mbox{for}\ \ s\in\{0,1\}.
\end{equation}

The $\mathbb{Z}_2$-gradation of $S(n)$, $\widetilde{S}(n)$ and
$H(n)$ is inherited from $W(n)$. Furthermore, the algebras $S(n)$
and $H(n)$ are $\mathbb{Z}$-graded subalgebras of $W(n)$. The
subalgebra $\widetilde{S}(n)$ is no longer $\mathbb{Z}$-graded. But
it has the following vector spaces direct sum decomposition in
$W(n)$:
\begin{equation}\label{2.2}
  \widetilde{S}(n)=\bigoplus_{j=-1}^{n-2}\widetilde{S}(n)_j,
\end{equation} where
$\widetilde{S}(n)_{-1}=\mbox{span}\{(1-\xi_1\xi_2\cdots\xi_n)\frac{\partial}{\partial\xi_i}|\
i=1,\ldots,n\}$ and $\widetilde{S}(n)_j=S(n)_j$ for $j>-1$.

The $0$-degree component of these superalgebras are Lie algebras.
Precisely, $W(n)_0\cong\ \mathfrak{gl}(n)$,
$S(n)_0=\widetilde{S}(n)_0\cong\ \mathfrak{sl}(n)$ and $H(n)_0\cong\
\mathfrak{so}(n)$. The other degree subspaces are all their
irreducible modules. In particular, the $(-1)$-degree subspace is a
nontrivial module.

The following proposition, called the \emph{transitivity} of $W(n)$,
$S(n)$ and $H(n)$, will be used in the Section 4.
\begin{proposition}{\bf(c.f.\cite{FSS})}
Let $\mathfrak{g}=W(n)$, $S(n)$ or $H(n)$. For any
$a\in\mathfrak{g}$, if $[a,\mathfrak{g}_{-1}]=0$ then $a\in
\mathfrak{g}_{-1}$.
\end{proposition}

\section{Proof of the Main Theorem for Classical Lie Superalgebras}
In this section, we first prove that
$\sigma|_{\mathfrak{g}_{\bar{0}}}=\mbox{id}$ for any Hom-Lie
superalgebra structure $(\mathfrak{g},\sigma)$ on a classical Lie
superalgebra $\mathfrak{g}$. Then thanks to the classification for
classical simple Lie superalgebras due to Kac \cite{K}, we can prove
the main theorem case by case.

\subsection{Hom-Lie algebra structure
on $\mathfrak{g}_{\bar{0}}$} The following lemma is clear by the
definition.

\begin{lemma} If $(\mathfrak{g},\sigma)$ is a Hom-Lie superalgebra
structure on the Lie superalgebra $\mathfrak{g}$, then
$(\mathfrak{g}_{\bar{0}},\sigma|_{\mathfrak{g}_{\bar{0}}})$ is a
Hom-Lie algebra structure on the Lie algebra
$\mathfrak{g}_{\bar{0}}$.
\end{lemma}

Below is a theorem obtained directly by main results in \cite{JL}.
It will be used to study the Hom-Lie superalgebra structures on a
classical Lie superalgebra.

\begin{theorem}{\bf (Jin-Li\cite{JL})}
If $\mathfrak{g}$ is a finite-dimensional semisimple Lie algebra and
$\sigma$ is an automorphism of $\mathfrak{g}$, then
$(\mathfrak{g},\sigma)$ is a Hom-Lie algebra if only if
$\sigma=\mbox{\emph{id}}$.
\end{theorem}
\begin{proof}
If $\mathfrak{g}$ is a finite-dimensional Lie algebra with
non-isomorphic simple summands, then $\sigma=\mbox{id}$ by the
Corollary 3.1 in \cite{JL}.

When $\mathfrak{g}$ has isomorphic simple summands, the Theorem 3.1
in \cite{JL} implies that for any non-trivial hom-Lie algebra
$(\mathfrak{g},\sigma)$ on $\mathfrak{g}$, the homomorphism $\sigma$
cannot be an automorphism because a projection homomorphism appears
as its factor.
\end{proof}

\begin{corollary}
For any classical simple Lie superalgebra $\mathfrak{g}$, if
$(\mathfrak{g},\sigma)$ is a Hom-Lie superalgebra, then
$\sigma|_{\mathfrak{g}_{\bar{0}}}=\mbox{\emph{id}}$.
\end{corollary}
\begin{proof}
For any classical simple Lie superalgebra $\mathfrak{g}$, its even
part $\mathfrak{g}_{\bar{0}}$ is a semisimple Lie algebra. Thus the
statement follows from Lemma 3.1 and Theorem 3.2.
\end{proof}

\subsection{Hom-Lie superalgebra structures on $\mathfrak{sl}(m|n), (m,n\geq1, m\not=n)$}
\

The special linear Lie superalgebra $\mathfrak{sl}(m|n)$ consists of
those matrices $\left(\begin{array}{cc}A&B\\C&D
\end{array}\right)\in\mathfrak{gl}(m|n)$ such that
$\mbox{tr}A-\mbox{tr}D=0$.

It has been shown in \cite{S,GP} that
$\textbf{Aut}\mathfrak{sl}(m|n)$ is generated by
$\mbox{Ad}(SL_m\times SL_n)$, $\jmath(\mathbb{C}^\times)$ and
$\tau$.

If $(\mathfrak{sl}(m|n),\sigma)$ is a Hom-Lie superalgebra, then
$\sigma|_{\mathfrak{sl}(m|n)_{\bar{0}}}=\mbox{id}$ by Corollary 3.3,
where $\mathfrak{sl}(m|n)_{\bar{0}}$ consists of those matrices
$\left(\begin{array}{cc}A&0\\0&D
\end{array}\right)\in\mathfrak{sl}(m|n)$. So it should be that $\sigma=\jmath(\lambda)$ for some
$\lambda\in\mathbb{C}^\times$ via (2.2)-(2.7) and (2.12)-(2.14).

Below we always use $e_{i,j}\in\mathfrak{gl}(m|n)$ to denote the
matrx whose $(i,j)$-th entry is $1$ and other entries are $0$. Set
$x=e_{1,m+1}$, $y=e_{1,2}$ and $z=e_{2,1}$ in \eqref{Jacobi}, we
have that $\lambda[x,[y,z]]=[x,[y,z]]$. Thus it must be that
$\lambda=1$ and hence $\sigma=\jmath(1)=\mbox{id}$.

\subsection{Hom-Lie superalgebra structure on $\mathfrak{psl}(n|n), (n>2)$}
\

Recall that
\begin{equation}
\mathfrak{psl}(n|n):=\mathfrak{sl}(n|n)/\{\lambda
I_{2n}|\lambda\in\mathbb{C}\}.
\end{equation} The automorphism group
$\textbf{Aut}\mathfrak{psl}(n|n)$ is generated by
$\mbox{Ad}(SL_n\times SL_n)$, $\jmath(\mathbb{C}^\times)$, $\tau$
and $\pi$ (c.f. \cite{S,GP}).

Also if $(\mathfrak{psl}(n|n),\sigma)$ is a Hom-Lie superalgebra,
then $\sigma=\jmath(\lambda)$ for some $\lambda\in\mathbb{C}^\times$
via Corollary 3.3, and, (2.2)-(2.7) and (2.12)-(2.17). Thus by the
arguments as the same as the case of $\mathfrak{sl}(m|n)$, we have
$\sigma=\jmath(1)=\mbox{id}$.

\subsection{Hom-Lie superalgebra structure on $\mathfrak{psl}(2|2)$}
\

It was proved in \cite{S,GP} that $\textbf{Aut}\mathfrak{psl}(2|2)$
is generated by $\mbox{Ad}(SL_2\times SL_2)$, $\rho(SL_2)$ and
$\pi$.

Now if $(\mathfrak{psl}(2|2),\sigma)$ is a Hom-Lie superalgebra,
then $\sigma=\rho\left(\begin{array}{cc}a&b\\c&d
\end{array}\right)$ for some $\left(\begin{array}{cc}a&b\\c&d
\end{array}\right)\in SL_2$ by Corollary 3.3, and, (2.2),(2.3),(2.8)-(2.11),(2.15),(2.18) and (2.19).

Now, let $x=e_{2,3}$, $y=e_{1,2}$ and $z=e_{2,1}$, then
$\sigma(x)=ae_{2,3}-ce_{4,1}$, $\sigma(y)=y$ and $\sigma(z)=z$. Thus
$[\sigma(x),[y,z]]=[x,[y,z]]$. This implies that $a=1$ and $c=0$.
Similarly, one sets $x=e_{3,2}$, $y=e_{3,4}$ and $z=e_{4,3}$, and
shows that $d=1$ and $c=0$ via $[\sigma(x),[y,z]]=[x,[y,z]]$ again.

\subsection{Hom-Lie superalgebra structure on $P(n-1)$}
\

The simple Lie superalgebra $P(n-1)$ is a subsuperalgebra of
$\mathfrak{sl}(n|n)$, consisting of the matrices of the following
form.
\begin{equation}
P(n-1)=\left\{\left(\begin{array}{cc}A&B\\C&-A^t\end{array}\right)|A\in\mathfrak{sl}_n,
B=B^t, C=-C^t\right\}
\end{equation}

There is a group homomorphism
\begin{equation}
\mbox{Ad}: SL_n\rightarrow \textbf{Aut}P(n-1),\quad\quad
X\mapsto\mbox{Ad}(X,(X^t)^{-1}).
\end{equation}

Automorphism group $\textbf{Aut}P(n-1)$ is generated by
$\mbox{Ad}(SL_n)$ and $\jmath(\mathbb{C}^\times)$.

A plausible automorphism $\sigma$ with $(P(n-1),\sigma)$ being a
Hom-Lie superalgebra,  should be that $\sigma=\jmath(\lambda)$ for
some $\lambda\in\mathbb{C}^\times$ by Corollary 3.3, and,
(2.2)-(2.5) and (2.12).

Let $x=e_{1,1}-e_{2,2}-e_{n+1,n+1}+e_{n+2,n+2}$, $y=e_{1,2}$ and
$z=e_{1,n+2}+e_{2,n+1}$. Then $\sigma(x)=x$, $\sigma(y)=y$ and
$\sigma(z)=\lambda z$. One has $\lambda[[x,y],z]=[x,y],z]$ by
$\sigma$-twisted Jacobi identity. Hence
$\lambda[e_{1,2},e_{1,n+2}+e_{2,n+1}]=[e_{1,2},e_{1,n+2}+e_{2,n+1}]$.
This shows that $\lambda=1$, i.e. $\sigma=\mbox{id}$.

\subsection{Hom-Lie superalgebra structure on $Q(n)$}
\

First we denote $\tilde{Q}(n-1)$ is the subsuperalgebra of
$\mathfrak{sl}(n|n)$ given by
\begin{equation}
\tilde{Q}(n-1):=\{\left(\begin{array}{cc}A&B\\B&A\end{array}\right)|\mbox{tr}B=0\}.
\end{equation}
The simple Lie superalgebra $Q(n-1)$ is the quotient
\begin{equation}
Q(n-1):=\tilde{Q}(n-1)/\{\lambda I_{2n}|\lambda\in\mathbb{C}\}.
\end{equation}

This simple Lie superalgebra is not invariant under the
supertransposition $\tau$, but it is so under the
$q$-supertransposition
\begin{equation}
\sigma_q:\left(\begin{array}{cc}A&B\\B&A\end{array}\right)\mapsto\left(\begin{array}{cc}A^t&\zeta
B^t\\\zeta B^t&A^t\end{array}\right),
\end{equation} where $\zeta$ is a fixed primitive $4$-th root of
unity.

There is a group homomorphism
\begin{equation}
\mbox{Ad}: SL_n\rightarrow \textbf{Aut}Q(n-1),\quad\quad
X\mapsto\mbox{Ad}(X,X).
\end{equation}

Automorphism group $\textbf{Aut}Q(n-1)$ is generated by
$\mbox{Ad}(SL_n)$ and $\sigma_q$. We can check easily that
\begin{equation}
\sigma_q^2=\jmath(-1)\quad \mbox{and}\quad \sigma_q^4=1,
\end{equation}
and
\begin{equation}
\sigma_q\mbox{Ad}(X)=\mbox{Ad}((X^t)^{-1})\sigma_q.
\end{equation}

Thus a plausible automorphism $\sigma$ with $(P(n-1),\sigma)$ being
a Hom-Lie superalgebra is $\sigma=\mbox{id}$ or $\sigma_q^2$ by
Corollary 3.3, and, (2.2),(2.3),(3.8) and (3.9). We can show that
$\sigma=\mbox{id}$ by setting $x=e_{1,2}+e_{n+1,n+2}$,
$y=e_{1,1}-e_{2,2}+e_{n+1,n+1}-e_{n+2,n+2}$ and
$z=e_{2,n+1}+e_{n+2,1}$ in the $\sigma$-twisted Jacobi identity
\eqref{Jacobi}.

\subsection{Hom-Lie superalgebra structure on $\mathfrak{osp}(m|2n)$}
Recall that the orthosymplectic Lie superalgebra
$\mathfrak{osp}(m|2n)$ is the subsuperalgebra of
$\mathfrak{sl}(m|2n)$ defined by
\begin{equation}
\mathfrak{osp}(m|2n):=\left\{\left(\begin{array}{cc}A&B\\J_nB^t&D\end{array}\right)|A\in\mathfrak{so}_m,B\in
M_{m,2n},D\in \mathfrak{sp}_{2n}\right\}
\end{equation}where
$J_n=\left(\begin{array}{cc}0&I_n\\-I_n&0\end{array}\right)$.

When $m$ is even, we take $\gamma_m\in O_m$ such that
$\det\gamma_m=-1$ and $\gamma_m^2=I_m$.

The automorphism group $\textbf{Aut}\mathfrak{osp}(m|2n)$ is
generated by $\mbox{Ad}(SO_m\times Sp_{2n})$ if $m$ is odd, and by
$\mbox{Ad}(SO_m\times Sp_{2n})$ and $\mbox{Ad}(\gamma_m,I_{2n})$ if
$m$ is even.

Clearly, $\sigma=\mbox{id}$ if $(\mathfrak{osp}(m|2n),\sigma)$ is a
Hom-Lie superalgebra by Corollary 3.3.

\subsection{Hom-Lie superalgebra structure on $G(3)$}
\

For Lie superalgebra $G(3)$, its even part $G(3)_{\bar{0}}\simeq
G_2\oplus\mathfrak{sl}_2$ and its automorphism group
$\textbf{Aut}G(3)$ is generated by $\mbox{Ad}(G_2\times SL_2)$.
Hence $\sigma=\mbox{id}$ if $(G(3),\sigma)$ is a Hom-Lie
superalgebra by Corollary 3.3.

\subsection{Hom-Lie superalgebra structure on $F(4)$}
\

For Lie superalgebra $F(4)$, its even part $F(4)_{\bar{0}}\simeq
\mathfrak{so}_7\oplus\mathfrak{sl}_2$ and its automorphism group
$\textbf{Aut}F(4)$ is generated by $\mbox{Ad}(\mbox{Spin}_7\times
SL_2)$. Hence $\sigma=\mbox{id}$ if $(F(4),\sigma)$ is a Hom-Lie
superalgebra by Corollary 3.3.

\subsection{Hom-Lie superalgebra structure on $D(2,1,\alpha)$}
\

For Lie superalgebra $\mathfrak{g}=D(2,1,\alpha)$, its even part is
$\mathfrak{g}_{\bar{0}}\simeq\mathfrak{sl}_2\oplus\mathfrak{sl}_2\oplus\mathfrak{sl}_2$,
and its odd part is $\mathfrak{g}_{\bar{1}}\simeq V_2\otimes
V_2\otimes V_2$, where $V_2$ is the natural module of
$\mathfrak{sl}_2$.

Fix $\sigma\in \mathfrak{S}_3$ and $\lambda\in \mathbb{C}^{\times}$.
Define $\theta(\sigma,\lambda)\in GL(D(2,1,\alpha))$ by
\begin{equation}
  \theta(\sigma,\lambda)((x_1,x_2,x_3), (u_1\otimes u_2\otimes
  u_3))=((x_{\sigma(1)},x_{\sigma(2)},x_{\sigma(3)}), \lambda(u_{\sigma(1)}\otimes u_{\sigma(2)}\otimes
  u_{\sigma(3)})).
\end{equation}
 Thus, if $\alpha\notin \{1,-\frac{1}{2},-2\}$ and $\alpha^3\neq1$, $\textbf{Aut}
 D(2,1,\alpha)$ is generated by $\mbox{Ad}(SL_2\times SL_2\times SL_2)$; if
 $\alpha\in\{1,-\frac{1}{2},-2\}$, $\textbf{Aut}
 D(2,1,\alpha)$ is generated by $\mbox{Ad}(SL_2\times SL_2\times SL_2)$ and
 $\theta((1,2),1)$; if $\alpha^3=1$ and $\alpha\neq1$, $\textbf{Aut}
 D(2,1,\alpha)$ is generated by $\mbox{Ad}(SL_2\times SL_2\times SL_2)$ and
 $\theta((1,2,3)\lambda)$, where $\lambda^2=\frac{1}{\alpha}$. These show
 that the Hom-Lie superalgebra construct over $D(2,1,\alpha)$ is
 trivial.

\section{Proof of the Main Theorem for Cartan Type Lie Superlagebras}

\subsection{Hom-Lie superalgebra structure on $W(n)$}
\

Let $x=\frac{\partial}{\partial\xi_i}$,
$y=\frac{\partial}{\partial\xi_j}$ and
$z=\xi_j\frac{\partial}{\partial\xi_l}$. Then $[x,y]=[x,z]=0$ if
$i\neq j$, and $[y,z]=\frac{\partial}{\partial\xi_l}$. The
$\sigma$-twisted Jacobi identity implies that
\begin{equation}\label{4.5}
  [\sigma(\frac{\partial}{\partial\xi_i}),\frac{\partial}{\partial\xi_l}]=0\quad\mbox{for any $l$}.
\end{equation} Since $W(n)$ is transitive by proposition 2.1, we have
$\sigma(W(n)_{-1})=W(n)_{-1}$. Set
$x=\frac{\partial}{\partial\xi_i}$,
$y=\xi_s\frac{\partial}{\partial\xi_t}$ and
$z=\xi_p\frac{\partial}{\partial\xi_q}$, where $s\neq i$ and $p\neq
i$. Then $[x,y]=[x,z]=0$ and
$[y,z]=\delta_{t,p}\xi_s\frac{\partial}{\partial\xi_q}-\delta_{s,q}\xi_p\frac{\partial}{\partial\xi_t}$.
Hence $[\sigma(x),[y,z]]=0$. We write
$\sigma(\frac{\partial}{\partial\xi_i})=\sum_{k=1}^na_{ki}\frac{\partial}{\partial\xi_k}$.
It follows that $a_{ki}=0$ if $k\neq i$, i.e.
$\sigma(\frac{\partial}{\partial\xi_i})=a_i\frac{\partial}{\partial\xi_i}$
for $i=1,\ldots,n$.

Now let $x=\xi_i\frac{\partial}{\partial\xi_j}$,
$y=\frac{\partial}{\partial\xi_k}$ and
$z=\xi_k\frac{\partial}{\partial\xi_l}$, where $k\neq j$, $l\neq i$
and $k\neq i$. Then $[x,y]=[x,z]=0$. It implies that
\begin{equation}\label{4.6}
[\sigma(\xi_i\frac{\partial}{\partial\xi_j}),\frac{\partial}{\partial\xi_l}]=0\quad\mbox{for
any $l\neq i$}.
\end{equation}
 For $l=i$, we have
\begin{equation}\label{4.7}
  [\sigma(\xi_i\frac{\partial}{\partial\xi_j}),\frac{\partial}{\partial\xi_i}]=
  [\sigma(\frac{\partial}{\partial\xi_k}),-\xi_{k}\frac{\partial}{\partial\xi_j}]
  =-a_k\frac{\partial}{\partial\xi_j}
  =a_k[\xi_i\frac{\partial}{\partial\xi_j},\frac{\partial}{\partial\xi_i}].
\end{equation} Thus
\begin{equation}\label{4.8}
  \sigma(\xi_i\frac{\partial}{\partial\xi_j})=a_k\xi_i\frac{\partial}{\partial\xi_j},
\end{equation} which also tells us that $a_1=a_2=\cdots=a_n$.
It follows that $\sigma|_{W(n)_0}=\mbox{id}$. Then $a_k=1$ for
$k=1,2,\ldots,n$ in \eqref{4.8}. Hence
\begin{equation}\label{4.8.1}
  \sigma|_{W(n)_{-1}\bigoplus W(n)_0}=\mbox{id}.
\end{equation}

Set $y\in W(n)_{-1}$ and $z\in W(n)_0$. By the $\sigma$-twisted
Jacobi identity we have that
\begin{equation}\label{4.9}
  [\sigma(x)-x,[y,z]]=0\quad \mbox{for any $x\in W(n)_l$},
\end{equation} where $l=1,2,\ldots,n-1$. Thus $\sigma(x)-x\in
W(n)_{-1}$ by the transitivity of $W(n)$. Finally, we set $y\in
W(n)_{0}$ and $z\in W(n)_0$ then get that $\sigma(x)=x$, i.e.
$\sigma=\mbox{id}$.

\subsection{Hom-Lie superalgebra structure on $S(n)$}
\

Since $S(n)_{-1}=W(n)_{-1}$ and $S(n)_0\subset W(n)_0$, and
$S(n)_0\cong \mathfrak{sl}(n)$, we obtain that
\begin{equation}\label{4.8.2}
\sigma|_{S(n)_{-1}\bigoplus S(n)_0}=\mbox{id}.
\end{equation}  Note that $S(n)$
is also transitive. Now the same argumentation as in the last
paragraph of above subsection says that $\sigma=\mbox{id}$.

\subsection{Hom-Lie superalgebra structure on $\widetilde{S}(n)$}
\

In this subsection, we always set $n(\geq4)$ is an even integer.
Denote $A=1-\xi_1\cdots\xi_n$, and
$A_i=\xi_1\cdots\xi_{i-1}\xi_{i+1}\cdots\xi_n$. Note that
$\widetilde{S}(n)=\bigoplus_{l=-1}^{n-2}\widetilde{S}(n)_l$ where
$\widetilde{S}(n)_{-1}=\mbox{span}\{A\frac{\partial}{\partial\xi_i}|i=1,\ldots,n\}$
and $\widetilde{S}(n)_l=S(n)_l$ for $l\geq0$.

Let $x=A\frac{\partial}{\partial\xi_i}$,
$y=A\frac{\partial}{\partial\xi_j}$ and
$z=\xi_p\frac{\partial}{\partial\xi_q}$, where $i\neq j$ and $p\neq
q$. Then
$[x,y]=(-1)^iA_i\frac{\partial}{\partial\xi_j}+(-1)^jA_j\frac{\partial}{\partial\xi_i}$,
$[x,z]=\delta_{p,i}A\frac{\partial}{\partial\xi_q}$ and
$[y,z]=\delta_{p,j}A\frac{\partial}{\partial\xi_q}$. Thus
\begin{equation}\label{4.10}
  [\sigma(A\frac{\partial}{\partial\xi_i}),\delta_{p,j}A\frac{\partial}{\partial\xi_q}]
  =[(-1)^iA_i\frac{\partial}{\partial\xi_j}+(-1)^jA_j\frac{\partial}{\partial\xi_i},
  \sigma(\xi_p\frac{\partial}{\partial\xi_q})]-
  [\sigma(A\frac{\partial}{\partial\xi_j}),\delta_{p,i}A\frac{\partial}{\partial\xi_q}],
\end{equation} via the $\sigma$-twisted Jacobi identity. Now let $p=j$. We
write $\sigma(A\frac{\partial}{\partial\xi_i})=\alpha_i+\beta_i$
where $\alpha_i\in\widetilde{S}(n)_{-1}$ and
$\beta_i\in\bigoplus_{l\geq1}\widetilde{S}(n)_l$. Then the equation
\eqref{4.10} becomes to be
\begin{equation}\label{4.11}
[\alpha_i+\beta_i,A\frac{\partial}{\partial\xi_q}]=
[(-1)^iA_i\frac{\partial}{\partial\xi_j}+(-1)^jA_j\frac{\partial}{\partial\xi_i},
\sigma(\xi_j\frac{\partial}{\partial\xi_q})].
\end{equation} Observe that $[\alpha_i,A\frac{\partial}{\partial\xi_q}]\in
W(n)_{n-2}$ and
$[(-1)^iA_i\frac{\partial}{\partial\xi_j}+(-1)^jA_j\frac{\partial}{\partial\xi_i},
\sigma(\xi_j\frac{\partial}{\partial\xi_q})]\in W(n)_{n-2}$. Thus
$[\beta_i,A\frac{\partial}{\partial\xi_q}]$ lies also in
$W(n)_{n-2}$. But
$\beta_i\in\bigoplus_{l=1}^{n-2}\widetilde{S}(n)_l$. In these
subspaces, the action of operator $A\frac{\partial}{\partial\xi_q}$
is just the same as $\frac{\partial}{\partial\xi_q}$, which implies
that $[\beta_i,\frac{\partial}{\partial\xi_q}]=0$ for $q\neq j$.
These show that $\beta_i=0$, which follows that
\begin{equation} \label{4.11.1}
\sigma(\widetilde{S}(n)_{-1})=\widetilde{S}(n)_{-1}.
\end{equation}
Now let $x=A\frac{\partial}{\partial\xi_i}$,
$y=\xi_s\frac{\partial}{\partial\xi_t}$ and
$z=\xi_p\frac{\partial}{\partial\xi_q}$, where $s\neq i$ and $p\neq
i$. Then $[x,y]=[x,z]=0$ and
$[y,z]=\delta_{p,t}\xi_s\frac{\partial}{\partial\xi_q}-\delta_{q,s}\xi_p\frac{\partial}{\partial\xi_t}$.
It follows that
\begin{equation}\label{4.12}
 [\sigma(A\frac{\partial}{\partial\xi_i})-A\frac{\partial}{\partial\xi_i},
 \delta_{p,t}\xi_s\frac{\partial}{\partial\xi_q}-\delta_{q,s}\xi_p\frac{\partial}{\partial\xi_t}]=0.
\end{equation} It shows that
\begin{equation}\label{4.13}
\sigma(A\frac{\partial}{\partial\xi_i})=a_iA\frac{\partial}{\partial\xi_i}\quad\mbox{for}\quad
i=1,\ldots,n.
\end{equation}

Set $x=\xi_i\frac{\partial}{\partial\xi_j}$,
$y=A\frac{\partial}{\partial\xi_k}$ and
$z=\xi_k\frac{\partial}{\partial\xi_l}$, where $k\neq j$ and $k\neq
i$. If $l\neq i$, we have $[x,y]=[x,z]=0$ and
$[y,z]=A\frac{\partial}{\partial\xi_l}$. Hence
$[\sigma(\xi_i\frac{\partial}{\partial\xi_j}),A\frac{\partial}{\partial\xi_l}]=0$
via the $\sigma$-twisted Jacobi identity. If $l=i$, we have
$[x,y]=0$, $[y,z]=A\frac{\partial}{\partial\xi_i}$ and
$[x,z]=-\xi_k\frac{\partial}{\partial\xi_j}$. This implies that
\begin{equation}\label{4.14}
[\sigma(\xi_i\frac{\partial}{\partial\xi_j}),A\frac{\partial}{\partial\xi_i}]=
[\sigma(A\frac{\partial}{\partial\xi_k}),-\xi_k\frac{\partial}{\partial\xi_j}]
=-a_kA\frac{\partial}{\partial\xi_j}
=a_k[\xi_i\frac{\partial}{\partial\xi_j},A\frac{\partial}{\partial\xi_i}].
\end{equation}
Thus
\begin{equation}\label{4.15}
 [\sigma(\xi_i\frac{\partial}{\partial\xi_j})-a_k\xi_i\frac{\partial}{\partial\xi_j},
 A\frac{\partial}{\partial\xi_l}]=0\quad \mbox{for any $l=1,\ldots,n$}.
\end{equation} Hence
$\sigma(\xi_i\frac{\partial}{\partial\xi_j})=a_k\xi_i\frac{\partial}{\partial\xi_j}$.
Since $\widetilde{S}(n)_0\cong \mathfrak{sl}(n)$, we have that
\begin{equation}\label{4.15.1}
  \sigma|_{\widetilde{S}(n)_{-1}\bigoplus\widetilde{S}(n)_0}=\mbox{id}.
\end{equation}

Now let $x\in\widetilde{S}(n)_l$ for $l>0$. Then
\begin{equation}\label{4.16}
  [\sigma(x),[y,z]]=[x,[y,z]]\quad \mbox{for
  $y\in\widetilde{S}(n)_{-1}$ and
$z\in\widetilde{S}(n)_{0}$}.
\end{equation} Thus
$\sigma(x)-x\in\widetilde{S}(n)_{-1}$. Since $\widetilde{S}(n)_{-1}$
is an nontrivial irreducible module of $\widetilde{S}(n)_{0}\cong
\mathfrak{sl}(n)$, one shows that $\sigma(x)=x$ via setting
$y\in\widetilde{S}(n)_{0}$ and $z\in\widetilde{S}(n)_{0}$. In a
word, $\sigma=\mbox{id}$ on the superalgebras $\widetilde{S}(n)$.

\subsection{Hom-Lie superalgebra structure on $H(n)$}
\

First, one notes that $H(n)_{-1}=W(n)_{-1}$ and $H(n)_0\cong
\mathfrak{so}(n)$.

Let $x=\frac{\partial}{\partial\xi_i}$,
$y=\frac{\partial}{\partial\xi_j}$ and
$z=\xi_j\frac{\partial}{\partial\xi_l}-\xi_l\frac{\partial}{\partial\xi_j}$,
where $i\neq j$, $i\neq l$ and $j\neq l$. Then $[x,y]=[x,z]=0$ and
$[y,z]=\frac{\partial}{\partial\xi_l}$. We obtain that
\begin{equation}\label{4.17}
  [\sigma(\frac{\partial}{\partial\xi_i}),\frac{\partial}{\partial\xi_l}]
  =[\frac{\partial}{\partial\xi_i},\frac{\partial}{\partial\xi_l}]=0
\end{equation} for $l\neq i$.

Now set $l=i$, i.e. $x=\frac{\partial}{\partial\xi_i}$,
$y=\frac{\partial}{\partial\xi_j}$ and
$z=\xi_j\frac{\partial}{\partial\xi_i}-\xi_i\frac{\partial}{\partial\xi_j}$.
We have that
\begin{equation}\label{4.18}
  [\sigma(\frac{\partial}{\partial\xi_i}),
  [\frac{\partial}{\partial\xi_j},\xi_j\frac{\partial}{\partial\xi_i}-\xi_i\frac{\partial}{\partial\xi_j}]]
  =-[\sigma(\frac{\partial}{\partial\xi_j}),
  [\frac{\partial}{\partial\xi_i},\xi_j\frac{\partial}{\partial\xi_i}-\xi_i\frac{\partial}{\partial\xi_j}]],
  \end{equation}i.e.
\begin{equation}\label{4.19}
[\sigma(\frac{\partial}{\partial\xi_i}),\frac{\partial}{\partial\xi_i}]=
[\sigma(\frac{\partial}{\partial\xi_j}),\frac{\partial}{\partial\xi_j}].
  \end{equation}
We write $\sigma(\frac{\partial}{\partial\xi_i})=\alpha_i+\beta_i$
where $\alpha_i\in H(n)_{-1}$ and $\beta_i\in
\bigoplus\limits_{l\geq1}H(n)_l$. Moreover, $\beta_i$ can be written
as $\beta_i=\sum\limits_{k=1}^n\frac{\partial
g_i}{\partial\xi_k}\frac{\partial}{\partial\xi_k}$ for
$i=1,\ldots,n$. Thus
\begin{equation}\label{4.20}
 [\sum\limits_{k=1}^n\frac{\partial
g_i}{\partial\xi_k}\frac{\partial}{\partial\xi_k},\frac{\partial}{\partial\xi_i}]=
[\sum\limits_{k=1}^n\frac{\partial
g_j}{\partial\xi_k}\frac{\partial}{\partial\xi_k},\frac{\partial}{\partial\xi_j}].
\end{equation}This gives that
\begin{equation}\label{4.21}
  [\frac{\partial
g_i}{\partial\xi_j}\frac{\partial}{\partial\xi_j},\frac{\partial}{\partial\xi_i}]=
[\frac{\partial
g_j}{\partial\xi_j}\frac{\partial}{\partial\xi_j},\frac{\partial}{\partial\xi_j}]=0.
\end{equation} Thus
\begin{equation}\label{4.22}
  [\beta_i,\frac{\partial}{\partial\xi_i}]=0.
\end{equation}It implies that $\sigma(H(n)_{-1})=H(n)_{-1}$.
Moreover, let $x=\frac{\partial}{\partial\xi_i}$,
$y=\xi_s\frac{\partial}{\partial\xi_t}-\xi_t\frac{\partial}{\partial\xi_s}$
and
$z=\xi_p\frac{\partial}{\partial\xi_q}-\xi_q\frac{\partial}{\partial\xi_p}$
with $s,t,p,q\neq i$. Then $[\sigma(x),[y,z]]=0$. We write
$\sigma(\frac{\partial}{\partial\xi_i})=\sum\limits_{k=1}^na_{ik}\frac{\partial}{\partial\xi_k}$.Then
\begin{equation}\label{4.23}
[\sum\limits_{k=1}^na_{ik}\frac{\partial}{\partial\xi_k},
\delta_{t,p}(\xi_s\frac{\partial}{\partial\xi_q}-\xi_q\frac{\partial}{\partial\xi_s})
-\delta_{s,q}(\xi_p\frac{\partial}{\partial\xi_t}-\xi_t\frac{\partial}{\partial\xi_p})
-\delta_{t,q}(\xi_s\frac{\partial}{\partial\xi_p}-\xi_p\frac{\partial}{\partial\xi_s})
+\delta_{s,p}(\xi_q\frac{\partial}{\partial\xi_t}-\xi_t\frac{\partial}{\partial\xi_q})]=0
\end{equation}for $s,t,p,q\neq i$. It shows that
\begin{equation}\label{4.24}
\sigma({\frac{\partial}{\partial\xi_i}})=a_i\frac{\partial}{\partial\xi_i}\quad\mbox{for
$i=1,\ldots,n$}.
\end{equation}

Let
$x=\xi_i\frac{\partial}{\partial\xi_j}-\xi_j\frac{\partial}{\partial\xi_i}$,
$y=\frac{\partial}{\partial\xi_s}$ and
$z=\xi_s\frac{\partial}{\partial\xi_t}-\xi_t\frac{\partial}{\partial\xi_s}$.
Then
\begin{equation}\label{4.25}
 [\sigma(\xi_i\frac{\partial}{\partial\xi_j}-\xi_j\frac{\partial}{\partial\xi_i}),
 \frac{\partial}{\partial\xi_t}]=0\quad\mbox{for $t\neq i$ and $t\neq j$}.
\end{equation}

Let
$x=\xi_i\frac{\partial}{\partial\xi_j}-\xi_j\frac{\partial}{\partial\xi_i}$,
$y=\frac{\partial}{\partial\xi_s}$ and
$z=\xi_s\frac{\partial}{\partial\xi_i}-\xi_i\frac{\partial}{\partial\xi_s}$
with $s\neq i$. Then
\begin{eqnarray}\label{4.26}
  [\sigma(\xi_i\frac{\partial}{\partial\xi_j}-\xi_j\frac{\partial}{\partial\xi_i}),
  \frac{\partial}{\partial\xi_i}]&=&
  [\sigma(\frac{\partial}{\partial\xi_s}),
  -(\xi_s\frac{\partial}{\partial\xi_j}-\xi_j\frac{\partial}{\partial\xi_s})]\\\nonumber
  &=&a_s[\xi_s\frac{\partial}{\partial\xi_j}-\xi_j\frac{\partial}{\partial\xi_s},\frac{\partial}{\partial\xi_s}]
  \\\nonumber
  &=&-a_s\frac{\partial}{\partial\xi_j}\\\nonumber
  &=&[a_s(\xi_i\frac{\partial}{\partial\xi_j}-\xi_j\frac{\partial}{\partial\xi_i}),
  \frac{\partial}{\partial\xi_i}]
\end{eqnarray} for $s\neq i$. We also show that
\begin{equation}\label{4.27}
[\sigma(\xi_i\frac{\partial}{\partial\xi_j}-\xi_j\frac{\partial}{\partial\xi_i}),
  \frac{\partial}{\partial\xi_i}]=
  [a_s(\xi_i\frac{\partial}{\partial\xi_j}-\xi_j\frac{\partial}{\partial\xi_i}),
  \frac{\partial}{\partial\xi_j}]\quad\mbox{for $s\neq j$}.
\end{equation} Hence
\begin{equation}\label{4.28}
 [\sigma(\xi_i\frac{\partial}{\partial\xi_j}-\xi_j\frac{\partial}{\partial\xi_i})-
a_s(\xi_i\frac{\partial}{\partial\xi_j}-\xi_j\frac{\partial}{\partial\xi_i}),\frac{\partial}{\partial\xi_k}
]=0 \quad\mbox{for any $k=1,\ldots,n$}.
\end{equation}
These tell us that $\sigma(H(n)_{0})=H(n)_{0}$. Thus
$\sigma|_{H(n)_{0}}=\mbox{id}$. Since
$\sigma(\xi_i\frac{\partial}{\partial\xi_j}-\xi_j\frac{\partial}{\partial\xi_i})=
a_s(\xi_i\frac{\partial}{\partial\xi_j}-\xi_j\frac{\partial}{\partial\xi_i})$,
we have $a_s=1$. Then
\begin{equation} \label{4.29}
\sigma|_{H(n)_{-1}\bigoplus H(n)_{0}}=\mbox{id}.
\end{equation}

For $x\in H(n)_l$, where $l>0$, we have
\begin{equation}\label{4.30}
 [\sigma(x)-x,[y,z]]=0
\end{equation} whenever $y\in H(n)_{-1}$ and $z\in H(n)_0$. This
shows that $\sigma(x)-x\in H(n)_{-1}$, and hence $\sigma(x)-x=0$ by
setting $y,z \in H(n)_0$. That is, $\sigma=\mbox{id}$.

\end{document}